\documentclass[ECP]{ejpecp} 

\usepackage{color}

\newcommand{\lsup}[1]{\underset{#1\to\infty}{\overline{\lim}}}

\definecolor{darkgreen}{rgb}{0,0.5,0}

\SHORTTITLE{Spectral Spin Glass Dynamics}

\TITLE{An Autonomous Spectral Representation for Fully Asymmetric Spherical Sherrington–Kirkpatrick Dynamics}

\AUTHORS{%
  James~MacLaurin\footnote{New Jersey Institute of Technology United States of America.
    \EMAIL{maclauri@njit.edu}}
  \and 
  Pedro~Vilanova\footnote{Stevens Institute of Technology,
    USA. \BEMAIL{pedro.vilanova@gmail.com} }}

\KEYWORDS{Spin Glass, Asymmetric, Spectral, Sherrington-Kirkpatrick, Neural Networks} 

\AMSSUBJ{NA} 

\SUBMITTED{July 1 2025} 
\ACCEPTED{December 13, 2025} 

\VOLUME{0}
\YEAR{2025}
\PAPERNUM{0}
\DOI{10.1214/YY-TN}

\ABSTRACT{
We study the large-system dynamics of a fully asymmetric spherical
Sherrington--Kirkpatrick model, allowing for a distinguished random signal
that interacts with the evolving spin configuration.  We first derive the
large-$N$ dynamical mean-field limit and characterize the law of a typical
coordinate by an effective Gaussian process whose covariance must satisfy a
self-consistency relation.  The resulting formulation is intrinsically
nonlocal in time and is not, in its usual form, a closed finite-dimensional
dynamical system.

Our main contribution is an autonomous spectral representation of this
limiting dynamics.  We introduce an auxiliary stationary chain of Gaussian
replicas whose one-step marginal reproduces the effective field appearing in
the dynamical mean-field equation.  Applying a matrix-valued Herglotz
representation, or equivalently Fourier analysis in the replica index,
converts the covariance structure of the chain into a family of
finite-dimensional matrix-valued evolution equations indexed by the unit
circle.  The macroscopic observables, including the signal overlap, are then
recovered from explicit spectral integrals.  This representation replaces the
original two-time self-consistency problem by an autonomous spectral closure
and provides a convenient framework for both analysis and numerical
computation.  We prove the validity of the representation and compare its
predictions with direct simulations of the finite-dimensional system.
 }

\begin{document}


\section{Introduction}

High-dimensional systems with quenched random interactions arise in several areas of probability, statistical physics, theoretical neuroscience, and mathematical ecology.  The classical Sherrington--Kirkpatrick (SK) model introduced a fully connected mean-field spin glass with random couplings \cite{SherringtonKirkpatrick1975}, while the spherical spin-glass model replaced the discrete spin constraint by a fixed-radius sphere and thereby produced a particularly tractable setting for dynamical questions \cite{KosterlitzThoulessJones1976}.  A parallel line of work, beginning with May's random-matrix model for ecosystem stability \cite{May1972}, emphasized that large random interaction matrices also arise naturally outside equilibrium statistical mechanics.  In neural-network models, the corresponding random couplings are typically directed rather than symmetric; their spectrum and their nonlinear dynamics have become central objects of study \cite{RajanAbbott2006,Sompolinsky1988}.  Related random-matrix stability questions continue to play an important role in ecology \cite{AllesinaTang2012}.

The distinction between symmetric and asymmetric couplings is substantial.  In the classical SK model the symmetry of the interaction matrix is tied to a Hamiltonian and to equilibrium methods.  When $J^{jk}$ and $J^{kj}$ are independent, as in the model considered here, the drift is no longer generated by a scalar Hamiltonian and there is in general no detailed-balance description.  This fully asymmetric regime was already analyzed for spherical Langevin systems by Crisanti and Sompolinsky \cite{CrisantiSompolinsky1987}, and it is closely related to the random recurrent-network model of Sompolinsky, Crisanti, and Sommers \cite{Sompolinsky1988}.  The latter work identified a transition from a stable regime to chaotic dynamics by deriving a self-consistent equation for the temporal autocorrelation.  We use the term \emph{asymmetric SK dynamics} for this SK-type mean-field Gaussian interaction structure without the reciprocity constraint $J^{jk}=J^{kj}$.

We study an $N$-dimensional spherical diffusion driven both by such an asymmetric random interaction matrix and by a random signal $X\in\mathbb R^N$.  For constants $\beta,q>0$, the dynamics are
\begin{multline}
 dZ^j_t = \bigg( qX^j +  \beta G^j_t - \beta N^{-1} Z^j_t\sum_{k=1}^N Z^k_t G^k_t
 - q N^{-1} Z^j_t \sum_{k=1}^N Z^k_t X^k
 - \sigma^2_t \frac{N-1}{2N} Z^j_t \bigg) dt \\
 + \sigma_t dW^j_t - N^{-1}\sigma_t Z^j_t \sum_{k=1}^N Z^k_t dW^k_t ,
 \label{eq: main dynamics intro}
\end{multline}
where
\begin{align}
 G^j_t = N^{-1/2}\sum_{k=1}^N J^{jk} Z^k_t.
\end{align}
The variables $\{J^{jk}\}_{1\leq j,k\leq N}$ are iid $\mathcal N(0,1)$, in particular with $J^{jk}$ independent of $J^{kj}$, and $(\sigma_t)_{t\geq0}$ is deterministic and continuous.  The signal $X$ is uniform on
\[
 \mathcal S^N:=\left\{x\in\mathbb R^N:\sum_{j=1}^N (x^j)^2=N\right\},
\]
and is independent of both the couplings and the Brownian motions.  The two projection terms in \eqref{eq: main dynamics intro} remove the radial components of the interaction and signal drifts, while the final two stochastic terms give the projected Brownian motion together with its It\^o correction.  Consequently, It\^o's formula yields
\begin{align}
 \sum_{j=1}^N (Z^j_t)^2 = \sum_{j=1}^N (Z^j_0)^2
\end{align}
for every $t\geq0$.  Thus the dynamics remain on the sphere once initialized there.

\subsection*{Relation to previous work}

The standard analytical description of large disordered dynamical systems is a dynamical mean-field theory: after averaging over the disorder, a typical coordinate is replaced by an effective stochastic process whose Gaussian forcing has a covariance that must be determined self-consistently from the law of that same process.  This point of view underlies the classical analyses of asymmetric spin systems and random neural networks \cite{CrisantiSompolinsky1987,Sompolinsky1988}; path-integral and generating-functional formulations provide a systematic language for the same reduction, and are reviewed for stochastic disordered systems in \cite{HertzRoudiSollich2017}.  The numerical scheme of Eissfeller and Opper \cite{eissfeller1992new} is an important early method for solving such self-consistent effective dynamics without simulating the original large system directly.

For spherical spin glasses with symmetric or partially symmetric interactions, two-time correlation and response functions lead to the Cugliandolo--Kurchan dynamical equations and to the study of aging and nonequilibrium relaxation \cite{CugliandoloKurchan1993}.  Rigorous work began with the large-deviation approach of Ben Arous and Guionnet \cite{BenArous1995}, which identifies the limiting path-space empirical measure through a self-consistent variational description.  The Cugliandolo--Kurchan equations were subsequently placed on a rigorous footing for mean-field spin-glass dynamics in \cite{BenArousDemboGuionnet2006}.  More recently, the universality of Langevin spin-glass dynamics have been determined by Dembo, Lubetzky, and Zeitouni \cite{Dembo2019}.  Related large-deviation and mean-field limits for randomly connected neural systems are developed, for example, in \cite{Faugeras2015,Cabana2018}, while broader treatments of the neural-network dynamical mean-field formalism may be found in \cite{Helias2020}.  These works make clear that the effective single-site description is powerful, but they also illustrate a recurring difficulty: the self-consistency is naturally expressed through two-time covariance and response kernels, so the resulting equations are typically nonlocal in time or otherwise awkward to analyze directly.

The signal term $qX$ places the model at the intersection of disordered dynamics and high-dimensional inference.  In planted models, one follows an overlap with a distinguished random direction and asks whether the dynamics amplify the initial correlation with that direction.  Closely related questions arise in the analysis of Langevin dynamics and first-order algorithms for tensor principal-component analysis and other mean-field inference problems \cite{Arous2020,Alaoui2020a}.  In the present model the observable
\[
 N^{-1}\sum_{j=1}^N Z_t^jX^j
\]
plays precisely this role.  The spherical constraint couples this overlap to the random interaction field, so a useful large-$N$ description must keep track simultaneously of signal alignment and interaction-induced correlations.

There are also several complementary strands of work on random networks that motivate the present setting.  In theoretical neuroscience, random asymmetric couplings produce irregular and chaotic activity, with the autocorrelation function serving as a principal macroscopic observable \cite{Sompolinsky1988,Cessac1994,Brunel2000,LerUrs06}. There is considerable interest in understanding how random recurrent neuronal networks can support low-dimensional structures \cite{mastrogiuseppe2018linking,huang2019circuit,clark2023dimension}.  Random-matrix structure is closely connected to the onset and organization of such dynamics \cite{RajanAbbott2006}.  In ecology, May's stability criterion and its later refinements likewise relate the spectrum and symmetry structure of large interaction matrices to collective stability \cite{May1972,AllesinaTang2012}.  The common mathematical theme is that microscopic randomness is high-dimensional, whereas the quantities of interest are low-dimensional self-consistent observables extracted from the limiting process.

\subsection*{Contribution of this paper}

Our first step is to identify the large-$N$ limit of the joint empirical law of the signal, the spin trajectory, and the random interaction field.  The limiting law is Gaussian and is characterized by a self-consistent effective stochastic equation.  This is in the spirit of the classical dynamical mean-field description and of the rigorous path-space limits above.  Since we have not found a proof in the literature for the precise spherical, fully asymmetric, signal-driven system \eqref{eq: main dynamics intro}, we include a proof tailored to this model.  The proof is based on a discretized comparison argument and also records the estimates needed later for the spectral representation.

The main new ingredient is a spectral reformulation of the limiting covariance.  The usual mean-field description is expressed through the two-time covariance of the effective Gaussian field, and this representation is not autonomous in a small collection of one-time observables.  We introduce instead an infinite stationary chain of replicas, in which each replica is driven by the preceding one together with independent white noise.  Translation invariance in the replica index makes the covariance kernel positive definite on $\mathbb Z$.  Herglotz's theorem then diagonalizes the shift in the replica coordinate and replaces the infinite covariance sequence by a spectral parameter $\theta\in\mathbb S^1$.  Frequency by frequency, the limiting dynamics become linear Gaussian equations whose covariance matrices satisfy a closed family of finite-dimensional ODEs, coupled only through explicit integrals over $\theta$.

This spectral variable should not be confused with an eigenvalue of the finite-$N$ coupling matrix $J$.  Indeed, a fully asymmetric Gaussian matrix is nonnormal, and the argument does not rely on diagonalizing $J$.  The spectral representation instead arises from the stationary replica shift.  This distinction is useful both conceptually and technically: the construction converts a nonlocal covariance problem into an autonomous evolution for a deterministic field of $4\times4$ covariance matrices.  In particular, the two scalar quantities that feed back into the limiting SDE,
\[
 \mathbb E[z_tg_t]
 \qquad\text{and}\qquad
 \mathbb E[z_tx],
\]
are recovered as explicit spectral integrals.  The resulting system is well suited to numerical integration and gives direct access to the distribution of covariance mass over replica frequencies.

The paper is organized as follows.  Section 2 states the limiting Gaussian law and the closed spectral covariance dynamics.  Section 3 discusses numerical implementation and compares the spectral prediction with finite-$N$ simulations.  Section 4 proves the large-$N$ limit, including the discretization and comparison estimates.  Section 5 constructs the stationary replica chain and proves the Herglotz representation that yields the autonomous spectral system.

\textit{Notation:} 

We let $\mathbf{1}: \mathbb{R} \to \mathbb{R}$ be the indicator function, i.e. $\mathbf{1}(x) = 1$ if $x=1$, else otherwise it equals $0$.
 
 \section{Outline of Model and Main Results}
 
We study $N$-dimensional spherical dynamics driven by a random signal $X \in \mathbb{R}^N$, i.e. for constants $q, \beta > 0$,
 \begin{multline}
 dZ^j_t = \bigg( qX^j +  \beta G^j_t - \beta N^{-1} Z^j_t\sum_{k=1}^N Z^k_t G^k_t  -  q N^{-1} Z^j_t \sum_{k=1}^N Z^k_t X^k- \sigma^2_t \frac{N-1}{2N} Z^j_t \bigg) dt \\+ \sigma_t dW^j_t - \frac{\sigma_t}{N}Z^j_t \sum_{k=1}^N Z^k_t dW^k_t , \label{eq: main dynamics}
 \end{multline}
 where
 \begin{align}
 G^j_t = N^{-1/2}\sum_{k=1}^N J^{jk} Z^k_t,
 \end{align}
 and $\lbrace J^{jk} \rbrace_{1\leq j,k \leq N}$ are iid $\mathcal{N}(0,1)$ Gaussian random variables, with $J^{jk} \neq J^{kj}$, and $(\sigma_t)_{t\geq 0}$ is deterministic, with $t\mapsto\sigma_t$ continuous. $X$ is randomly selected according to the uniform distribution over $\mathcal{S}^N$. It is known that $X$ may be easily sampled as follows. Select $\lbrace \tilde{X}^j \rbrace_{j\in I_N}$ to be independent Gaussian $\mathcal{N}(0,1)$, and then define
 \begin{align}
 X^j :=\sqrt{N} \tilde{X}^j \big( \sum_{k\in I_N} (\tilde{X}^k)^2 \big)^{-1/2}.
 \end{align}
We take the initial condition to have a specified level of overlap with the random signal. Accordingly, we define the initial condition to be such that 
\begin{equation}
Z^j_0 := \sqrt{N} \big(\tilde{q} X^j +\sqrt{1-\tilde{q}^2} \tilde{Z}^j_0 \big) \bigg\lbrace \sum_{k=1}^N \big(\tilde{q} X^k +\sqrt{1-\tilde{q}^2} \tilde{Z}^k_0 \big)^2 \bigg\rbrace^{-1/2}.
\end{equation}
 Here $\tilde{q} \in [0,1]$ is a constant,  $(\tilde{Z}^j_0)_{j\in I_N}$ are iid $\mathcal{N}(0,1)$, and $\mathbb{E}[\tilde{Z}^j_0 X^k] = 0$. Notice that
 \begin{align}
 \sum_{j \in I_N} (Z^j_0)^2 = N.
 \end{align}
 Consider the triple empirical measure
 \begin{align}
 \hat{\mu}^N_T = N^{-1}\sum_{j=1}^N \delta_{X^j , Z^j_{[0,T]} , G^j_{[0,T]}} \in \mathcal{P}\big( \mathbb{R} \times \mathcal{C}([0,T],\mathbb{R}^2) \big).
 \end{align}
 The following theorem characterizing the large $N$ limit of asymmetric spin glass dynamics is generally known \cite{Helias2020}. Recently a proof was provided in \cite{MacLaurin2024} in the case that $q = 0$ and the dynamics is not spherical. Nevertheless, because there does not seem to exist a proof for the exact system considered in this paper, we provide a quick proof in Section \ref{sec:proof}.
 \begin{theorem} \label{Theorem Main Limit}
 There exists a centered Gaussian probability law $\xi$ such that, with unit probability, $\hat{\mu}^N_T \to \xi$. Here $\xi$ is the unique law of random variables $\big( x , z_{[0,T]} ,g_{[0,T]}\big)$, with the following covariance structure. Let $(x , y , W_{[0,T]} , g_{[0,T]})$ be independent centered Gaussian variables: with $x\sim \mathcal{N}(0,1)$, $y\sim \mathcal{N}(0,1-\tilde{q}^2)$ and $W_{[0,T]}$ a Brownian Motion. Conditionally on the variables 
 $(x , y , W_{[0,T]} , g_{[0,T]})$, define $z \in \mathcal{C}([0,T],\mathbb{R})$ to be the strong solution to the following stochastic differential equation,
 \begin{align}
 dz_t = \bigg(qx - q z_t\mathbb{E}^{\xi}[z_t x] + \beta g_t - \beta z_t \mathbb{E}^{\xi}[ z_t g_t] - \frac{\sigma_t^2}{2} z_t \bigg)dt + \sigma_t dW_t,
 \end{align}
with $z_0 = \tilde{q}x + y$. The process $g_{[0,T]}$ is centered and Gaussian, with covariance structure
 \begin{align}
 \mathbb{E}[ g_s g_t ] = \mathbb{E}[z_s z_t].
 \end{align}
 \end{theorem}
 Classically, the limiting measure $\xi$ has been studied using delay differential equations \cite{Sompolinsky1988, BenArous1995,ben2018spectral,Helias2020}. The main result of this paper is to determine an autonomous spectral representation of the dynamics. The dynamics is parameterized by $\mathbb{S}^1 = (-\pi,\pi]$.
 
 We now specify the spectral dynamics.  The construction is interpreted frequency by frequency. For each fixed $\theta\in\mathbb S^1$, let $a(\theta),b(\theta)$ be independent $\mathcal N(0,1/2)$ variables and let $w(\theta),\tilde w(\theta)$ be independent Brownian motions, independent of $a(\theta),b(\theta)$. A simultaneous measurable realization for all $\theta$ is not needed; the deterministic covariance field below is the object used in the theorem. Define
\begin{multline}
du_t(\theta)=\bigg(q a(\theta)+\beta\cos(\theta)u_t(\theta)-\beta\sin(\theta)v_t(\theta)-\frac{\sigma_t^2}{2}u_t(\theta)\\
-\frac{\beta}{2\pi}u_t(\theta)\int_{\mathbb S^1}\big(\mathbb E[u_t(\alpha)^2]+\mathbb E[v_t(\alpha)^2]\big)\cos(\alpha)\,d\alpha
-\frac{q}{2\pi}u_t(\theta)\int_{\mathbb S^1}\mathbb E[u_t(\alpha)a(\alpha)+v_t(\alpha)b(\alpha)]\,d\alpha\bigg)dt\\
+2^{-1/2}\sigma_t\,dw_t(\theta),
\label{eq: ut theta dynamics}
\end{multline}
\begin{multline}
dv_t(\theta)=\bigg(q b(\theta)+\beta\sin(\theta)u_t(\theta)+\beta\cos(\theta)v_t(\theta)-\frac{\sigma_t^2}{2}v_t(\theta)\\
-\frac{\beta}{2\pi}v_t(\theta)\int_{\mathbb S^1}\big(\mathbb E[u_t(\alpha)^2]+\mathbb E[v_t(\alpha)^2]\big)\cos(\alpha)\,d\alpha
-\frac{q}{2\pi}v_t(\theta)\int_{\mathbb S^1}\mathbb E[u_t(\alpha)a(\alpha)+v_t(\alpha)b(\alpha)]\,d\alpha\bigg)dt\\
+2^{-1/2}\sigma_t\,d\tilde w_t(\theta).
\label{eq: vt theta dynamics}
\end{multline}
The initial conditions are
\begin{align}
u_0(\theta)&=\tilde q a(\theta)+\hat u_0(\theta),\\
v_0(\theta)&=\tilde q b(\theta)+\hat v_0(\theta),
\end{align}
where, for each fixed $\theta$, $\hat u_0(\theta),\hat v_0(\theta)$ are independent centered Gaussian variables, independent of $a(\theta),b(\theta)$, with
\begin{align}
\mathbb E[\hat u_0(\theta)^2]&=(1-\tilde q^2)/2,\\
\mathbb E[\hat v_0(\theta)^2]&=(1-\tilde q^2)/2,\\
\mathbb E[\hat u_0(\theta)\hat v_0(\theta)]&=0.
\end{align}
Writing $y_t(\theta) = \big( u_t(\theta) , v_t(\theta), a(\theta), b(\theta) \big)^T \in \mathbb{R}^4$, define the covariance matrix
\begin{equation}
\mathcal{K}_t(\theta) = \mathbb{E}\big[ y_t(\theta) y_t(\theta)^T \big] \in \mathbb{R}^{4 \times 4}.
\end{equation}We can determine a self-contained set of ODEs to solve for $\big\lbrace \mathcal{K}_t(\theta) \big\rbrace_{\theta\in\mathbb{S}^1}$ as follows. 
\begin{lemma}
\label{lemma22}
There exists a unique solution $ \mathcal{K}_t \in \mathcal{C}\big([0,T], \mathcal{C}(\mathbb{S}^1) \big)$ to the following system, such that for each $\theta \in \mathbb{S}^1$,
\begin{align} \label{eq: K theta evolution}
\frac{d}{dt}\mathcal{K}_t(\theta)  = B_t(\theta) \mathcal{K}_t(\theta) + \mathcal{K}_t(\theta) B_t(\theta)^T + \frac{1}{2} \left(\begin{array}{c c c c}
\sigma_t^2 & 0 & 0 & 0 \\
0 & \sigma_t^2 & 0 & 0\\
0 & 0 & 0 & 0 \\
0 & 0 & 0 & 0
\end{array}\right)
\end{align}
and
\begin{align}
B_t(\theta) = \left(\begin{array}{c c c c}
- \sigma_t^2 / 2 -  ( \beta A_t +  qC_t) + \beta \cos(\theta) & -\beta \sin(\theta) & q & 0 \\
\beta \sin(\theta) & - \sigma_t^2 /  2 - ( \beta A_t + q C_t) + \beta \cos(\theta) & 0 & q \\
0 & 0 & 0 & 0 \\
0 & 0 & 0 & 0
\end{array}\right).
\end{align}
and
\begin{align}
A_t =& \frac{1}{2\pi} \int_{\mathbb{S}^1} \big( \mathcal{K}_{t,11}(\theta)+\mathcal{K}_{t,22}(\theta) \big)\cos(\theta)\,d\theta, \\
C_t=& \frac{1}{2\pi}\int_{\mathbb{S}^1} \big( \mathcal{K}_{t,13}(\theta) + \mathcal{K}_{t,24}(\theta) \big)\,d\theta .
\end{align}
The initial condition is, for $\theta \in \mathbb{S}^1$,
\begin{align}
\label{eq:init}
\mathcal{K}_0(\theta) = \left(\begin{array}{ c c c c}
1/2 & 0 & \tilde{q}/2 & 0 \\
0 & 1/2 & 0 & \tilde{q}/2  \\
 \tilde{q}/2 & 0 & 1/2 & 0 \\
0 & \tilde{q}/2  & 0 & 1/2
\end{array}\right)\,.
\end{align}
\end{lemma}
In the next theorem, we demonstrate that the spectral system corresponds to the covariance of the limiting Gaussian law $\xi$ of Theorem \ref{Theorem Main Limit}. This result is significant because it provides an autonomous description of out-of-equilibrium dynamics in high dimensional recurrent neuronal networks.
\begin{theorem} \label{Theorem Herglotz}
For all $t\geq 0$,
\begin{align}
\mathbb{E}^{\xi}[ z_t g_t] &= A_t, \\
\mathbb{E}^{\xi}[z_t x] &= C_t .
\end{align}
\end{theorem}

\section{Discussion and Numerical Implementation}

We performed a variety of numerical experiments to illustrate the convergence of the stochastic system to the spectral representation. These results illustrate the convergence as $N\to\infty$. The spectral dynamics also reveals the presence of multiple timescales. There is often an initial transient dynamics, followed by exponential convergence to equilibrium. 

\subsection{Details on Numerical Implementation}
All numerical experiments use the closed spectral covariance system of Lemma~\ref{lemma22}, with the observables $A_t=\mathbb E_\xi[z_tg_t]$ and $C_t=\mathbb E_\xi[z_tx]$ identified in Theorem~\ref{Theorem Herglotz}. In the baseline simulation, used for Figures~\ref{fig:At_validation}--\ref{fig:rho_spectral}, we take $\beta=1.20$, $q=0.60$, $\tilde q=0.40$, final time $T=3$, and constant noise $\sigma_t\equiv 1$. The spectral ODE is integrated by a fourth-order Runge--Kutta scheme on a uniform time grid with $n_t=1500$ steps and on a uniform Fourier grid on $(-\pi,\pi]$ with $n_\theta=256$ modes. The initial covariance is the matrix $\mathcal{K}_0(\theta)$ from \eqref{eq:init}, so that the initial signal-state cross-covariances are $K_{13}(0,\theta)=K_{24}(0,\theta)=\tilde q/2=0.20$. To validate the spectral prediction against the original finite-$N$ spherical dynamics, we simulate the particle system by Euler--Maruyama with reprojection onto the sphere at every step, using $N=1200$, $40$ Monte Carlo repetitions, and time step $\Delta t_{\mathrm{MC}}=5\times 10^{-4}$. For the convergence study in $N$, we use $N\in\{100,200,400,800,1600,3200,6400,12800\}$, with $30$ Monte Carlo repetitions for each system size, time step $\Delta t_{\mathrm{conv}}=2.5\times 10^{-3}$, and $95\%$ pointwise normal-approximation confidence intervals for the mean discrete-time sup-errors. The selected time step was checked to be sufficiently small relative to the finite-size error. In addition, we perform a long-time spectral study for Figures~\ref{fig:kinf_beta1} and \ref{fig:kinf_beta10} with $q=0.10$, $\tilde q=0.10$, $T=12$, $n_t=6000$, $n_\theta=256$, and $\sigma_t\equiv 1$, for $\beta\in\{1,10\}$, in order to examine the late-time spectral profiles.

\begin{figure}[t]
    \centering
    \includegraphics[width=0.82\linewidth]{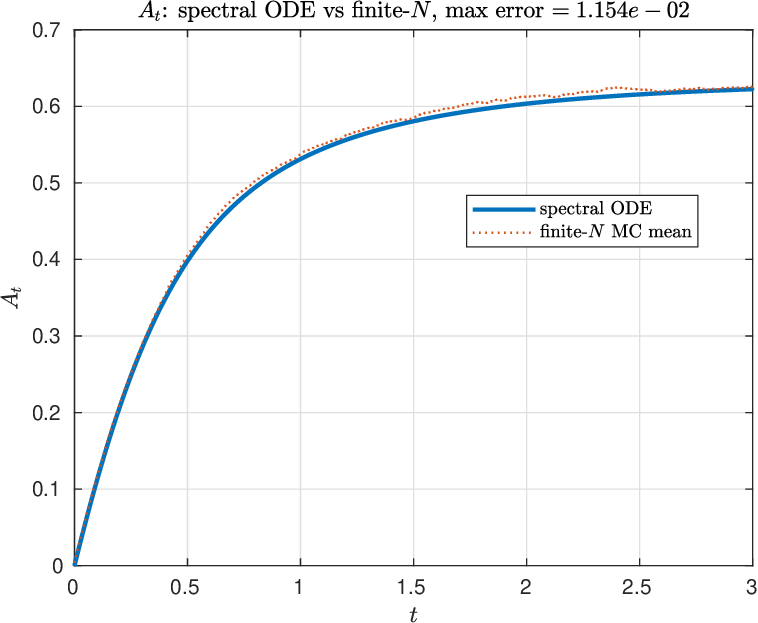}
    \caption{Comparison of the spectral prediction for $A_t=\mathbb E_\xi[z_tg_t]$ with its empirical finite-$N$ counterpart $N^{-1}\sum_{j=1}^N Z_t^jG_t^j$. The solid curve is the solution of the spectral covariance ODE, while the dotted curve is the Monte Carlo mean obtained from Euler--Maruyama simulations of the spherical particle system with $N=1200$. Here the reported max error is the quantity $\max_{t_k}\lvert A^{\mathrm{spec}}(t_k)-A^{\mathrm{MC}}(t_k)\rvert$, namely the largest absolute difference between the spectral curve and the Monte Carlo mean over the discrete set of sampled times used in the comparison.}
    \label{fig:At_validation}
\end{figure}

\begin{figure}[t]
    \centering
    \includegraphics[width=0.82\linewidth]{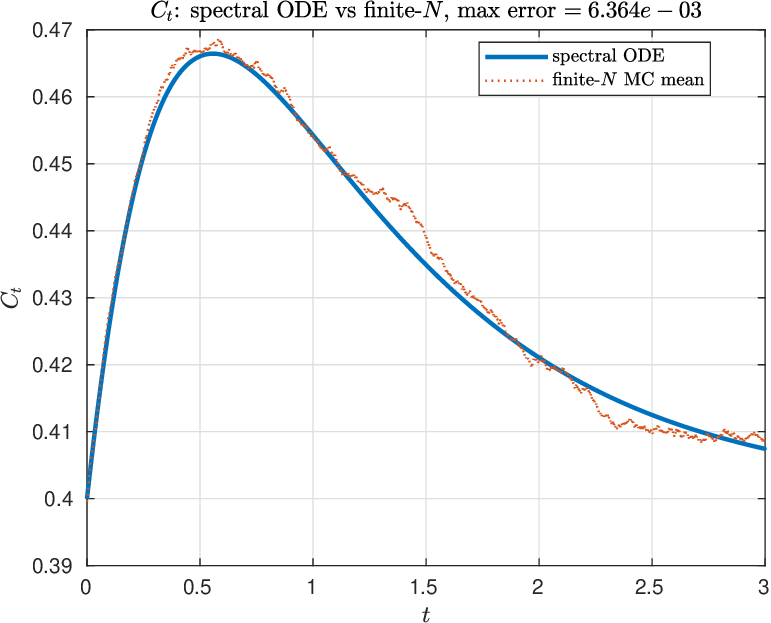}
    \caption{Comparison of the spectral prediction for $C_t=\mathbb E_\xi[z_tx]$ with its empirical finite-$N$ counterpart $N^{-1}\sum_{j=1}^N Z_t^jX^j$. The solid curve represents the solution of the spectral covariance ODE, and the dotted curve is the Monte Carlo mean from the finite-$N$ particle simulations. The reported max error is defined as $\max_{t_k}\lvert C^{\mathrm{spec}}(t_k)-C^{\mathrm{MC}}(t_k)\rvert$, that is, the largest absolute difference between the two curves over the sampled time grid.}
    \label{fig:Ct_validation}
\end{figure}

\begin{figure}[t]
    \centering
    \includegraphics[width=0.82\linewidth]{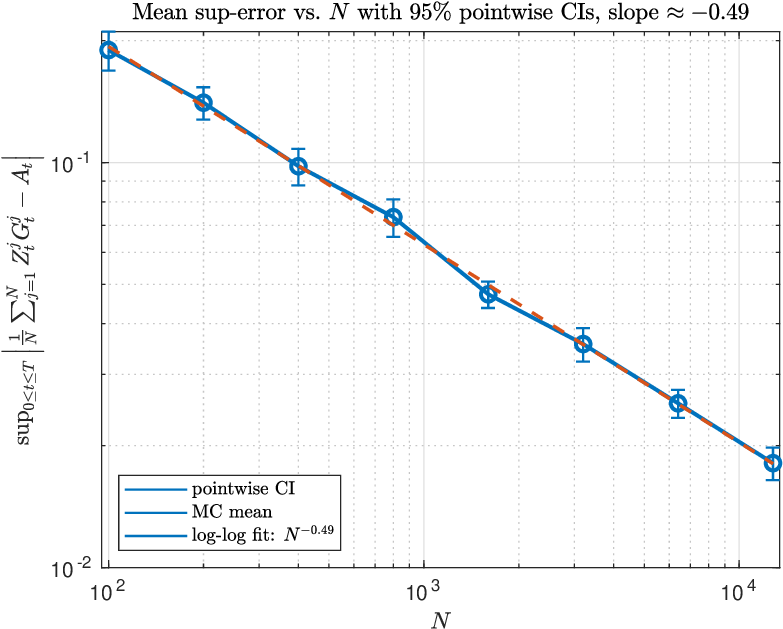}
    \caption{Finite-size convergence of the interaction observable. The figure reports the mean discrete-time sup-error $\sup_{0\le t\le T}\left|\frac{1}{N}\sum_{j=1}^N Z_t^jG_t^j-A_t\right|$ as a function of $N$, together with pointwise confidence intervals and a log--log regression fit. The approximately linear behavior in log--log scale is consistent with a power-law decay of the discrepancy between the particle system and its spectral-limit prediction.}
    \label{fig:Nconv_A}
\end{figure}

\begin{figure}[t]
    \centering
    \includegraphics[width=0.82\linewidth]{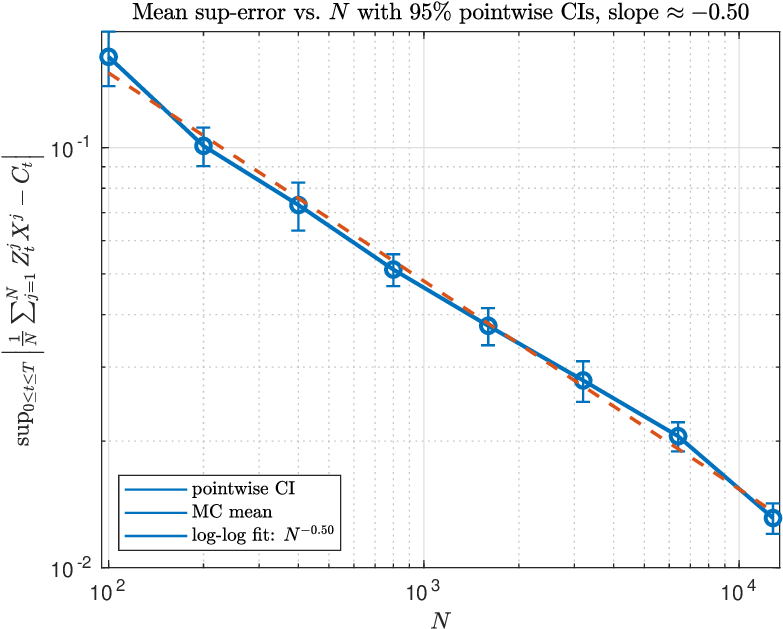}
    \caption{Finite-size convergence of the signal-overlap observable. The figure shows the mean discrete-time sup-error $\sup_{0\le t\le T}\left|\frac{1}{N}\sum_{j=1}^N Z_t^jX^j-C_t\right|$ versus $N$, with pointwise confidence intervals and a log--log fit. The decay of the error is consistent with the convergence of the empirical overlap with the signal to the quantity $C_t$ predicted by the spectral theory.}
    \label{fig:Nconv_C}
\end{figure}

\begin{figure}[t]
    \centering
    \includegraphics[width=0.82\linewidth]{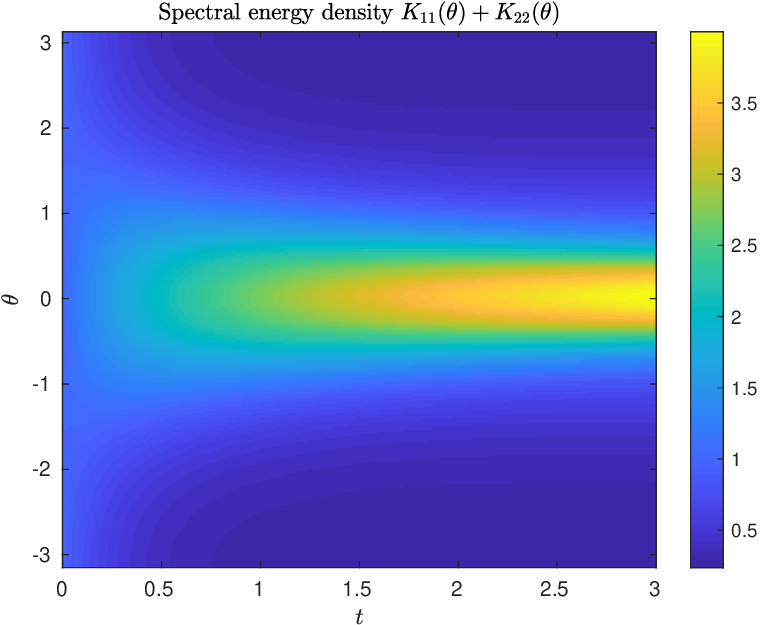}
    \caption{Spectral energy density $K_{11}(\theta)+K_{22}(\theta)$ as a function of time $t$ and Fourier variable $\theta$. The figure displays the evolution of the spectral covariance under the closed ODE system of Lemma~\ref{lemma22}, for the baseline parameter choice $\beta=1.20$, $q=0.60$, $\tilde q=0.40$, and $\sigma_t\equiv 1$.}
    \label{fig:rho_spectral}
\end{figure}


\begin{figure}[t]
    \centering
    \includegraphics[width=0.82\linewidth]{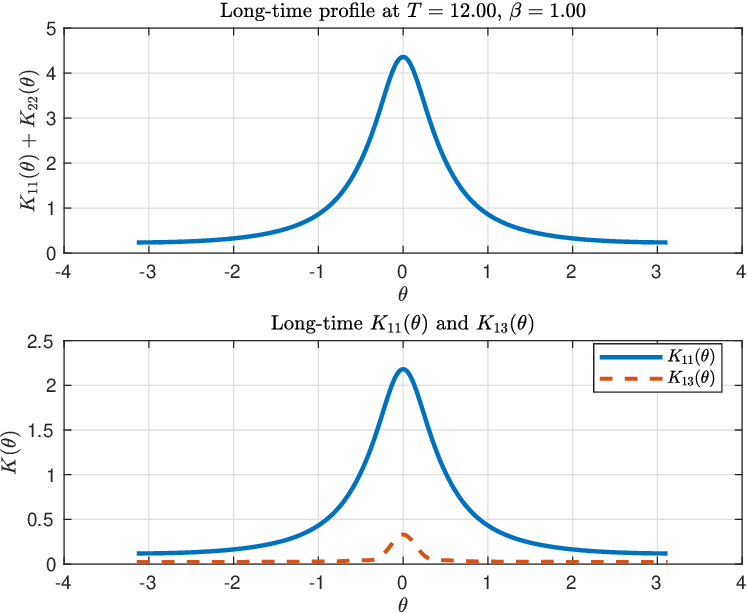}
    \caption{Long-time spectral profile at terminal time $T=12$ for $\beta=1.00$, $q=0.10$, $\tilde q=0.10$, and $\sigma_t\equiv 1$. The upper panel shows the late-time energy density $K_{11}(\theta)+K_{22}(\theta)$, while the lower panel compares $K_{11}(\theta)$ and $K_{13}(\theta)$, illustrating the late-time distribution of energy and signal correlation across Fourier modes.}
    \label{fig:kinf_beta1}
\end{figure}

\begin{figure}[t]
    \centering
    \includegraphics[width=0.82\linewidth]{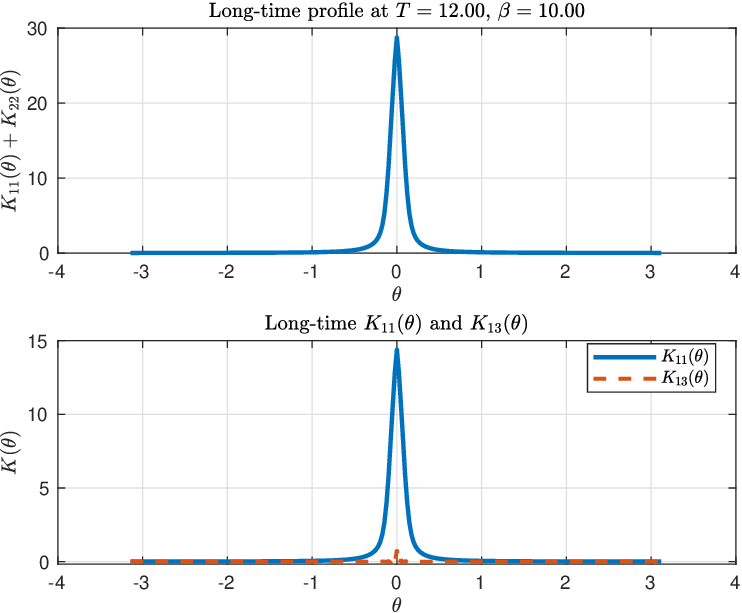}
    \caption{Long-time spectral profile at terminal time $T=12$ for $\beta=10.00$, $q=0.10$, $\tilde q=0.10$, and $\sigma_t\equiv 1$. Displayed in the same format as Figure~\ref{fig:kinf_beta1}, this plot shows how the long-time spectral organization changes when the coupling strength is increased.}
    \label{fig:kinf_beta10}
\end{figure}

\section{Proof of Theorem \ref{Theorem Main Limit}}
\label{sec:proof}

We give the proof in four steps. First we define a discrete-time approximation and
identify its large-$N$ limit. Second we compare the discrete particle system with the
spherical diffusion. Third we show that the discrete limiting law converges to the
continuous fixed point. Finally we combine these estimates by a diagonal argument.

Throughout this section, let
\[
\mathcal J_N=N^{-1/2}(J^{jk})_{j,k\leq N},
\qquad
P_z=I_N-\frac{zz^T}{N},\qquad |z|^2=N.
\]
In vector form, \eqref{eq: main dynamics} is
\begin{equation}
 dZ_t=P_{Z_t}(qX+\beta \mathcal J_N Z_t)\,dt
 -\frac{N-1}{2N}\sigma_t^2 Z_t\,dt
 +\sigma_tP_{Z_t}\,dW_t .
 \label{eq:vector-spherical}
\end{equation}
In particular, $|Z_t|^2=N$ for every $t$ almost surely.  We use $d_W$ for the
$2$-Wasserstein distance.  On the discrete state space
\[
E_m=\mathbb R\times \mathbb R^{m+1}\times \mathbb R^{m+1}
\]
we use the normalized norm
\begin{equation}
 \|(x,y,g)\|_m^2=|x|^2+\frac{1}{m+1}\sum_{a=0}^m\big(|y_a|^2+|g_a|^2\big).
 \label{eq:discrete-normalized-norm}
\end{equation}
All constants denoted by $C$ or $c$ below may depend on
$T,q,\beta$ and $\sup_{t\leq T}|\sigma_t|$, but not on $N$ or $m$.

\subsection{The discrete limiting law}

Fix $m\geq1$, put $\Delta=T/m$ and $t_a=a\Delta$.  Define the Euler system
\begin{align}
 y_{a+1}^j
  &=y_a^j+\Delta\bigg(qX^j+\beta g_a^j
   -\Big(qC_a^N+\beta A_a^N+\frac12\sigma_{t_a}^2\Big)y_a^j\bigg)
   +\sigma_{t_a}w_a^j, \label{eq:euler-particle}\\
 g_a^j&=(\mathcal J_Ny_a)^j,\qquad
 C_a^N=\frac1N\sum_{k=1}^Ny_a^kX^k,\qquad
 A_a^N=\frac1N\sum_{k=1}^Ny_a^kg_a^k, \label{eq:euler-observables}
\end{align}
where $w_a=W_{t_{a+1}}-W_{t_a}$ and $y_0=Z_0$.  Let
\begin{equation}
 \hat\mu^{(m),N}
 =\frac1N\sum_{j=1}^N\delta_{(X^j,y_{0:m}^j,g_{0:m}^j)}.
 \label{eq:discrete-empirical}
\end{equation}

We next define its deterministic candidate limit.  Let
$x,\eta,w_0,\ldots,w_{m-1}$ be independent centered Gaussian variables, with
$x,\eta\sim\mathcal N(0,1)$ and $w_a\sim\mathcal N(0,\Delta)$, and put
\begin{equation}
 y_0=\tilde q x+\sqrt{1-\tilde q^2}\,\eta.
 \label{eq:disc-initial}
\end{equation}
The centered Gaussian vector $(g_0,\ldots,g_m)$ is independent of
$(x,\eta,w_0,\ldots,w_{m-1})$ and is specified self-consistently by
\begin{equation}
 \mathbb E[g_ag_b]=\mathbb E[y_ay_b],\qquad 0\leq a,b\leq m,
 \label{eq:disc-g-cov}
\end{equation}
while
\begin{equation}
 y_{a+1}=y_a+\Delta\bigg(qx+\beta g_a
 -\Big(qC_a+\beta A_a+\frac12\sigma_{t_a}^2\Big)y_a\bigg)
 +\sigma_{t_a}w_a,
 \label{eq:disc-limit-update}
\end{equation}
where
\begin{equation}
 C_a=\mathbb E[y_ax],\qquad A_a=\mathbb E[y_ag_a].
 \label{eq:disc-AC}
\end{equation}
We denote the law of $(x,y_{0:m},g_{0:m})$ by $\xi^{(m)}$. For $0 \leq a \leq m$, the law of $(x,y_{0:a},g_{0:a})$ by $\xi^{(m,a)}$

\begin{lemma}\label{Lemma Limiting Laws Well Defined}
For every $m$, the law $\xi^{(m)}$ is uniquely determined.  The continuous law
$\xi$ in Theorem \ref{Theorem Main Limit} is also uniquely determined.  If
$\tilde\xi^{(m)}$ denotes the law of
$(x,z_{t_0:t_m},g_{t_0:t_m})$ under $\xi$, then
\begin{equation}
 d_W\big(\xi^{(m)},\tilde\xi^{(m)}\big)\longrightarrow0
 \qquad (m\to\infty),
 \label{eq:disc-to-cont-law}
\end{equation}
where $d_W$ on $E_m$ is associated with \eqref{eq:discrete-normalized-norm}.
\end{lemma}

\begin{proof}
The discrete law is constructed inductively.  Suppose that the joint law up to time
$a$ has been constructed.  Equation \eqref{eq:disc-limit-update} determines $y_{a+1}$,
hence all numbers $\mathbb E[y_{a+1}y_b]$, $0\leq b\leq a+1$.  There is a unique
centered Gaussian extension $g_{a+1}$ whose covariances are prescribed by
\eqref{eq:disc-g-cov}.  This proves existence and uniqueness of $\xi^{(m)}$.

For the continuous law it is convenient to work directly with covariances.  Set
\[
 Q(s,t)=\mathbb E[z_sz_t],\qquad C(t)=\mathbb E[z_tx],\qquad
 A(t)=\mathbb E[z_tg_t],
\]
and
\[
 \lambda(t)=qC(t)+\beta A(t)+\frac12\sigma_t^2,
 \qquad
 R(t,s)=\exp\left\{-\int_s^t\lambda(r)\,dr\right\},\qquad s\leq t.
\]
Variation of constants gives
\begin{equation}
 z_t=C(t)x+\sqrt{1-\tilde q^2}\,R(t,0)\eta
 +\beta\int_0^tR(t,r)g_r\,dr
 +\int_0^tR(t,r)\sigma_r\,dW_r,
 \label{eq:voc-z}
\end{equation}
where necessarily
\begin{align}
 C(t)&=\tilde q R(t,0)+q\int_0^tR(t,r)\,dr, \label{eq:C-volterra}\\
 A(t)&=\beta\int_0^tR(t,r)Q(r,t)\,dr, \label{eq:A-volterra}\\
 Q(s,t)&=C(s)C(t)+(1-\tilde q^2)R(s,0)R(t,0) \nonumber\\
 &\quad+\beta^2\int_0^s\int_0^tR(s,r)R(t,v)Q(r,v)\,dv\,dr
 +\int_0^{s\wedge t}R(s,r)R(t,r)\sigma_r^2\,dr.
 \label{eq:Q-volterra}
\end{align}
Conversely, any continuous solution of
\eqref{eq:C-volterra}--\eqref{eq:Q-volterra} determines through
\eqref{eq:voc-z} a Gaussian law satisfying the specification in Theorem
\ref{Theorem Main Limit}.  On a sufficiently short interval the right-hand side of
\eqref{eq:C-volterra}--\eqref{eq:Q-volterra} is a contraction in the uniform norm:
differences of $R$ are bounded by the interval length times the uniform differences of
$A$ and $C$, and every remaining difference occurs under at least one time integral. Thus the system has a unique local solution.

At a fixed point, It\^o's formula and \eqref{eq:Q-volterra} give
\begin{align}
 \frac{d}{dt}Q(t,t)
 &=2qC(t)+2\beta A(t)-2\lambda(t)Q(t,t)+\sigma_t^2 \nonumber\\
 &=\big(2qC(t)+2\beta A(t)+\sigma_t^2\big)\big(1-Q(t,t)\big).
 \label{eq:Q-diagonal}
\end{align}
Since $Q(0,0)=1$, it follows that
\begin{equation}
 Q(t,t)=1,\qquad 0\leq t\leq T.
 \label{eq:limit-spherical}
\end{equation}
Consequently $|A(t)|\leq1$ and $|C(t)|\leq1$ by Cauchy--Schwarz, so the local
construction can be continued to all of $[0,T]$.  This proves uniqueness of $\xi$.

Finally, the discrete covariance recursion obtained from
\eqref{eq:disc-limit-update} is the left-endpoint Euler discretization of
\eqref{eq:C-volterra}--\eqref{eq:Q-volterra}.  The bounds just obtained are uniform in
$m$.  Uniform continuity of $\sigma$ and a discrete Gronwall argument therefore give
\begin{equation}
 \max_{0\leq a,b\leq m}
 \Big(|C_a-C(t_a)|+|A_a-A(t_a)|
 +|\mathbb E[y_ay_b]-Q(t_a,t_b)|\Big)\longrightarrow0.
 \label{eq:covariance-euler-convergence}
\end{equation}
Centered Gaussian laws depend continuously in $W_2$ on their covariance matrices.
Using the norm in \eqref{eq:discrete-normalized-norm}, the uniform convergence
in \eqref{eq:covariance-euler-convergence} yields
\eqref{eq:disc-to-cont-law}.
\end{proof}

\begin{lemma}\label{Lemma Discretized Convergence}
For every fixed $m$ and every $\epsilon>0$ there is $c=c(m,\epsilon)>0$ such that
\begin{equation}
 \limsup_{N\to\infty}\frac1N\log
 \mathbb P\big(d_W(\hat\mu^{(m),N},\xi^{(m)})>\epsilon\big)\leq-c.
 \label{eq:discrete-exponential-convergence}
\end{equation}
\end{lemma}

\begin{proof}
Our proof proceeds inductively in time. Define the empirical measure upto timestep $a$:
\begin{equation}
 \hat\mu^{(m,a),N}
 =\frac1N\sum_{j=1}^N\delta_{(X^j,y_{0:a}^j,g_{0:a}^j)} \in \mathcal{P}\big( \mathbb{R} \times \mathbb{R}^{a+1} \times  \mathbb{R}^{a+1}\big).
 \label{eq:discrete-empirical a}
\end{equation}
The 2-Wasserstein Distance on $\mathcal{P}\big( \mathbb{R} \times \mathbb{R}^{a+1} \times  \mathbb{R}^{a+1}\big)$ is denoted by $d_W^{(a)}$.

We first claim that the empirical law of $(X^j,Z_0^j)_{j\leq N}$ is exponentially close to the law of
$(x,y_0)$: that is, for any $\epsilon > 0$, there exists $C_{\epsilon}$ such that
\begin{align}
\lsup{N} N^{-1}\log \mathbb{P}\bigg( d^{(0)}_W\big( \hat\mu^{(m,0),N}, \xi^{(m,0)} \big) \geq \epsilon \bigg) \leq - C_{\epsilon}.\label{inductive step zero}
\end{align}
Indeed, both variables are obtained from independent Gaussian samples by
normalization with empirical Euclidean norms, and those norms have exponentially small
deviations from one. This establishes \eqref{inductive step zero}.

We now make an inductive hypothesis: that the empirical law at time $a$ is exponentially-close to its limit. Define Hypothesis $(a)$ to be that for any $\epsilon > 0$, there exists a constant $C_{\epsilon}^a$ such that
\begin{align}
\lsup{N} N^{-1}\log \mathbb{P}\bigg( d^{(a)}_W\big( \hat\mu^{(m,a),N}, \xi^{(m,a)} \big) \geq \epsilon \bigg) \leq - C^a_{\epsilon}.\label{inductive step a}
\end{align}
Our next goal is to prove that, if Hypothesis $(a)$ holds, then necessarily Hypothesis $(a+1)$ holds too. To this end, let
$Y_a$ be the $N\times(a+1)$ matrix whose columns are $y_0,\ldots,y_a$, and put
\begin{equation}
 K_a^N=\frac1N Y_a^TY_a.
 \label{eq:empirical-Ka}
\end{equation}
Condition on the complete history $(X,y_{0:a},g_{0:a})$.  For a fixed row $j$, the
Gaussian vector $N^{-1/2}(J^{j1},\ldots,J^{jN})$ is then conditioned on the linear
constraints
\[
 N^{-1/2}\sum_{k=1}^NJ^{jk}y_b^k=g_b^j,\qquad 0\leq b\leq a.
\]
Since the rows of
$J$ are independent, the residual Gaussian variables must also be independent over $j$.

Suppose first that $K_a^N$ is invertible.  Put
\[
 k_{a+1}^N(b)=\frac1N\langle y_{a+1},y_b\rangle,
 \qquad
 s_{a+1}^N=\frac1N|y_{a+1}|^2.
\]
The standard conditional Gaussian formula gives
\begin{align}
 \mathbb E\big[g_{a+1}^j\mid X,y_{0:a+1},g_{0:a}\big]
  &=(k_{a+1}^N)^T(K_a^N)^{-1}g_{0:a}^j,
 \label{eq:cond-mean-gauss}\\
 \operatorname{Var}\big(g_{a+1}^j\mid X,y_{0:a+1},g_{0:a}\big)
  &=s_{a+1}^N-(k_{a+1}^N)^T(K_a^N)^{-1}k_{a+1}^N.
 \label{eq:cond-var-gauss}
\end{align}
These are exactly the conditional mean and variance of the Gaussian extension in
\eqref{eq:disc-g-cov}.  The update of $y_{a+1}^j$, conditional on the past, is also
Gaussian, with independent innovation $\sigma_{t_a}w_a^j$.  
We can now conclude that Hypothesis $(a+1)$ must hold using a conditional Sanov Theorem.


If $K_a^N$ is singular, the same formulas hold with the Moore--Penrose inverse and the
conditional Gaussian law is supported on the corresponding affine subspace.
\end{proof}

\subsection{Short-time regularity and comparison with the spherical system}

For $L>2$, define the operator-norm event
\begin{equation}
 \mathcal O_N(L)=\{\|\mathcal J_N\|_{\mathrm{op}}\leq L\}.
 \label{eq:operator-event}
\end{equation}
A standard bound gives a constant $c_L>0$ such that
\begin{equation}
 \mathbb P\big(\mathcal O_N(L)^c\big)\leq e^{-c_LN}
 \label{eq: bound max eigenvalue of random matrix}
\end{equation}
for all sufficiently large $N$.  

For $m\geq1$ define
\begin{equation}
 f^{(m),j}=\mathbf 1\left\{
 \max_{0\leq a<m}\sup_{0\leq h\leq\Delta}
 |Z_{t_a+h}^j-Z_{t_a}^j|\geq m^{-1/4}\right\}.
 \label{eq:fluctuation-indicator}
\end{equation}

Set
\begin{equation}
 \Omega_m^N=\frac1N\sum_{j=1}^N
 \max_{0\leq a<m}\sup_{0\leq h\leq\Delta}
 |Z_{t_a+h}^j-Z_{t_a}^j|^2.
 \label{eq:Omega-def}
\end{equation}
\begin{lemma}\label{Lemma Time Fluctuations of Original System}
For every $\epsilon>0$ there is $m_\epsilon$ such that, for every $m\geq m_\epsilon$,
\begin{equation}
 \limsup_{N\to\infty}\frac1N\log\mathbb P(\Omega_m^N>\epsilon)<0.
 \label{eq:Omega-exp}
\end{equation}
Moreover,
\begin{equation}
 \lim_{m\to\infty}\limsup_{N\to\infty}\mathbb E[\Omega_m^N]=0.
 \label{eq:Omega-mean}
\end{equation}
Consequently, for every fixed $\epsilon>0$, 
\begin{align}
\limsup_{m\to\infty}\limsup_{N\to\infty} N^{-1}\log \mathbb{P}\bigg(\sum_{j=1}^N f^{(m),j} \geq \epsilon \bigg) = -\infty.
\end{align}
\end{lemma}

\begin{proof}
On $\mathcal O_N(L)$, the spherical identities $|Z_t|=|X|=\sqrt N$ imply
\[
 |\mathcal J_NZ_t|\leq L\sqrt N,\qquad
 \frac1N|\langle Z_t,\mathcal J_NZ_t\rangle|\leq L,
 \qquad
 \frac1N|\langle Z_t,X\rangle|\leq1.
\]
Thus the drift in \eqref{eq:vector-spherical} has Euclidean norm at most $C\sqrt N$,
uniformly in time.  The martingale coefficient is $\sigma_tP_{Z_t}$ and
$\|P_{Z_t}\|_{\mathrm{op}}=1$.  The Burkholder--Davis--Gundy inequality, applied on
each interval $[t_a,t_{a+1}]$, together with the usual Brownian modulus estimate,
gives
\[
 \lim_{m\to\infty}\limsup_{N\to\infty}
 \mathbb E\big[\Omega_m^N\mathbf1_{\mathcal O_N(L)}\big]=0.
\]
Together with \eqref{eq: bound max eigenvalue of random matrix}, this proves
\eqref{eq:Omega-mean}.

For fixed $m$, after stopping the dynamics when the Brownian paths leave a large ball,
the map from the driving Gaussian paths to $\Omega_m^N$ is Lipschitz in the
Cameron--Martin norm, with squared Lipschitz constant of order $N^{-1}$.  Gaussian
concentration therefore yields, for every fixed $r>0$,
\[
 \mathbb P\Big(\Omega_m^N-\mathbb E[\Omega_m^N\mid J,X]>r,
 \mathcal O_N(L)\Big)\leq e^{-c(m,r)N}.
\]
The stopping may be removed using the Gaussian exponential moments of the Brownian
modulus.  Choosing $m$ so large that the conditional mean is below $\epsilon/2$ outside
an exponentially small exceptional event proves \eqref{eq:Omega-exp}.

If $N^{-1}\sum_jf^{(m),j}>\epsilon$, then
$\Omega_m^N\geq\epsilon m^{-1/2}$.  Combining this implication with the preceding
concentration estimate gives the assertion concerning the fraction of large coordinates.
Finally, on $\mathcal O_N(L)$,
\[
 \frac1N|G_t-G_s|^2
 =\frac1N|\mathcal J_N(Z_t-Z_s)|^2
 \leq L^2\frac1N|Z_t-Z_s|^2,
\]
which proves the corresponding statements for $G$.
\end{proof}

Let
\begin{equation}
 \bar\mu^{(m),N}
 =\frac1N\sum_{j=1}^N\delta_{(X^j,Z_{t_0:t_m}^j,G_{t_0:t_m}^j)}.
 \label{eq:grid-empirical}
\end{equation}

\begin{lemma}\label{Lemma main discretized convergence}
For every $\epsilon>0$ there exists $m_\epsilon$ such that, for $m\geq m_\epsilon$,
\begin{equation}
 \limsup_{N\to\infty}\frac1N\log
 \mathbb P\big(d_W(\bar\mu^{(m),N},\hat\mu^{(m),N})>\epsilon\big)<0.
 \label{eq:grid-euler-exponential}
\end{equation}
\end{lemma}

\begin{proof}
Couple the two systems with the same $X,J,W$ and initial condition and put
\[
 e_a=\frac1N|y_a-Z_{t_a}|^2.
\]
Integrating \eqref{eq:vector-spherical} from $t_a$ to $t_{a+1}$ and subtracting
\eqref{eq:euler-particle} gives an Euler consistency error, a projection-martingale
error, and terms that are Lipschitz in $y_a-Z_{t_a}$.  On $\mathcal O_N(L)$,
Cauchy--Schwarz gives
\begin{align*}
 \left|A_a^N-\frac1N\langle Z_{t_a},\mathcal J_NZ_{t_a}\rangle\right|
 &\leq C e_a^{1/2},\\
 \left|C_a^N-\frac1N\langle Z_{t_a},X\rangle\right|
 &\leq e_a^{1/2}.
\end{align*}
The time variation of these quantities within a mesh interval is controlled by
Lemma \ref{Lemma Time Fluctuations of Original System}.  The radial It\^o correction
in \eqref{eq:vector-spherical} differs from $-\sigma_{t_a}^2Z_{t_a}/2$ by
$O(N^{-1})+o_m(1)$ after integration over one mesh interval.

The martingale omitted by the Euler scheme is
\[
 M_t=-\int_0^t\frac{\sigma_s}{N}Z_sZ_s^T\,dW_s.
\]
Its total quadratic-variation trace is bounded by a constant depending only on $T$ and
$\sup|\sigma|$, because $|Z_s|^2=N$.  Hence
$N^{-1}\sup_{t\leq T}|M_t|^2$ is exponentially small at every fixed positive level.
A standard stopping argument, stopped at the first index for which $e_a>1$, keeps
$|y_a|/\sqrt N$ uniformly bounded while the estimate is being proved.  Combining the
preceding estimates and using $2uv\leq\delta u^2+\delta^{-1}v^2$ yields, outside an event
of probability at most $e^{-cN}$,
\begin{equation}
 \max_{0\leq a\leq m}e_a\leq C_T\big(r_m+N^{-1}\big),
 \qquad r_m\longrightarrow0,
 \label{eq: compare discretzed temp}
\end{equation}
where a discrete Gronwall inequality is used in the last step.  For $m$ and then $N$
sufficiently large, the right-hand side is strictly smaller than one, so the stopping time
is not reached.

By the coupling and the norm \eqref{eq:discrete-normalized-norm},
\begin{align*}
 d_W(\bar\mu^{(m),N},\hat\mu^{(m),N})^2
 &\leq\frac1{m+1}\sum_{a=0}^m e_a
 +\frac1{m+1}\sum_{a=0}^m\frac1N
 |\mathcal J_N(y_a-Z_{t_a})|^2\\
 &\leq(1+L^2)\max_{0\leq a\leq m}e_a
 \qquad\text{on }\mathcal O_N(L).
\end{align*}
The conclusion follows from \eqref{eq: compare discretzed temp} and
\eqref{eq: bound max eigenvalue of random matrix}.
\end{proof}

\subsection{Completion of the proof}

\begin{proof}[Proof of Theorem \ref{Theorem Main Limit}]
Let $\Pi_m$ be restriction of a continuous path to the grid $t_0,\ldots,t_m$, and let
$I_m$ denote piecewise-linear interpolation.  By the triangle inequality,
\begin{align}
 d_W(\hat\mu_T^N,\xi)
 &\leq d_W\big(\hat\mu_T^N,(I_m\Pi_m)_\#\hat\mu_T^N\big) \nonumber\\
 &\quad+d_W\big((I_m)_\#\bar\mu^{(m),N},(I_m)_\#\hat\mu^{(m),N}\big) \nonumber\\
 &\quad+d_W\big((I_m)_\#\hat\mu^{(m),N},(I_m)_\#\xi^{(m)}\big) \nonumber\\
 &\quad+d_W\big((I_m)_\#\xi^{(m)},\xi\big).
 \label{eq:final-triangle}
\end{align}
The first term tends to zero, with an exponentially small exceptional probability, by
Lemma \ref{Lemma Time Fluctuations of Original System} and the operator-norm bound
for $G$.  The second term is controlled by Lemma
\ref{Lemma main discretized convergence}.  For fixed $m$, the third term tends to zero
exponentially fast in $N$ by Lemma \ref{Lemma Discretized Convergence}.  The fourth
term tends to zero as $m\to\infty$ by Lemma
\ref{Lemma Limiting Laws Well Defined} and continuity of the limiting Gaussian paths.

Consequently, for every $\epsilon>0$ one can first choose $m$ and then $N_0$ so that
for $N\geq N_0$,
\[
 \mathbb P\big(d_W(\hat\mu_T^N,\xi)>\epsilon\big)\leq e^{-c_\epsilon N}.
\]
The right-hand side is summable in $N$.  Borel--Cantelli, applied to
$\epsilon=1,1/2,1/3,\ldots$, yields
\[
 d_W(\hat\mu_T^N,\xi)\longrightarrow0\qquad\text{almost surely}.
\]
The covariance characterization and Gaussianity of $\xi$ were established in
Lemma \ref{Lemma Limiting Laws Well Defined}; this completes the proof.
\end{proof}

\section{Proof of the Spectral Representation}

Set
\begin{equation}
 A_t=\mathbb E^\xi[z_tg_t],\qquad
 C_t=\mathbb E^\xi[z_tx],\qquad
 \lambda_t=\beta A_t+qC_t+\frac12\sigma_t^2.
 \label{eq:lambda-spectral}
\end{equation}
By \eqref{eq:limit-spherical}, $\mathbb E[z_t^2]=1$; hence $|A_t|\leq1$ and
$|C_t|\leq1$.

\subsection{A stationary chain of replicas}

Let $(x^p,\eta^p,W^p)_{p\in\mathbb Z}$ be independent copies of a standard normal, a
standard normal, and a Brownian motion, respectively.  Consider
\begin{equation}
 dz_t^p=\big(qx^p+\beta z_t^{p-1}-\lambda_t z_t^p\big)\,dt
 +\sigma_t\,dW_t^p,
 \qquad
 z_0^p=\tilde q x^p+\sqrt{1-\tilde q^2}\,\eta^p.
 \label{eq: infinite stochastic system}
\end{equation}

\begin{lemma}\label{lemma:replica-chain}
System \eqref{eq: infinite stochastic system} has a unique solution whose law is
stationary in the replica index $p$.  For every $p\in\mathbb Z$,
\begin{equation}
 \mathcal L\big(x^p,z_{[0,T]}^p,z_{[0,T]}^{p-1}\big)=\xi.
 \label{eq:replica-law-xi}
\end{equation}
\end{lemma}

\begin{proof}
Fix $\alpha>0$ and let
\[
 H_\alpha=
 \left\{u=(u^p)_{p\in\mathbb Z}:\sum_{p\in\mathbb Z}e^{-\alpha|p|}|u^p|^2<\infty\right\}.
\]
The left shift is a bounded operator on $H_\alpha$.  Since the weights are summable,
the iid initial sequence and the Brownian sequence define $H_\alpha$-valued random
variables and an $H_\alpha$-valued Brownian motion.  Equation
\eqref{eq: infinite stochastic system} is therefore a linear stochastic evolution equation
with bounded time-dependent coefficients, and it has a unique solution in
$\mathcal C([0,T];H_\alpha)$.

The data and the equation are invariant in law under shifts of $p$, so uniqueness implies
stationarity.  Moreover, iterating the variation-of-constants formula in the left-shift
term shows that the contribution to $z^p$ from an innovation at index $p-n$ is bounded,
on $[0,T]$, by a constant times $(|\beta|T)^n/n!$.  It follows in particular that
$z^{p-1}$ is measurable with respect to the innovations with indices at most $p-1$ and
is independent of $(x^p,\eta^p,W^p)$.  Stationarity gives
\[
 \mathbb E[z_s^{p-1}z_t^{p-1}]=\mathbb E[z_s^pz_t^p].
\]
Thus the triple $(x^p,z^p,z^{p-1})$ satisfies exactly the fixed-point specification of
Theorem \ref{Theorem Main Limit}: the third component is a centered Gaussian process,
independent of the innovations driving the second component, and its covariance equals
the covariance of the second component.  Uniqueness of $\xi$ from
Lemma \ref{Lemma Limiting Laws Well Defined} proves \eqref{eq:replica-law-xi}.
\end{proof}

\subsection{Matrix-valued Herglotz representation}

For each $t$, the sequence
\[
 U_t^p=(z_t^p,x^p)^T
\]
is stationary in $p$.  Let
\[
 \Gamma_t(k)=\mathbb E\big[U_t^{p+k}(U_t^p)^T\big],\qquad k\in\mathbb Z.
\]
The matrix-valued Herglotz theorem gives a positive Hermitian spectral measure on
$\mathbb S^1$.  In the present system the measure is absolutely continuous.  Indeed,
at time zero the covariance sequence is supported at lag zero, the forcing is spatially
white, and iteration of the left-shift term produces coefficients bounded by
$C_T^{|k|}/|k|!$; hence $\sum_k\|\Gamma_t(k)\|<\infty$ uniformly for $t$ in compact
intervals.  We write the resulting continuous spectral density as $S_t(\theta)$,
normalized by
\begin{equation}
 \Gamma_t(k)=\frac1{2\pi}\int_{-\pi}^{\pi}e^{ik\theta}S_t(\theta)\,d\theta.
 \label{eq:matrix-herglotz}
\end{equation}

At frequency $\theta$, \eqref{eq: infinite stochastic system} has complex drift matrix
\begin{equation}
 \mathcal B_t^{\mathrm c}(\theta)=
 \begin{pmatrix}
 -\lambda_t+\beta e^{i\theta} & q\\
 0&0
 \end{pmatrix}.
 \label{eq:complex-drift}
\end{equation}
Consequently,
\begin{equation}
 \frac{d}{dt}S_t(\theta)
 =\mathcal B_t^{\mathrm c}(\theta)S_t(\theta)
 +S_t(\theta)\mathcal B_t^{\mathrm c}(\theta)^*
 +\begin{pmatrix}\sigma_t^2&0\\0&0\end{pmatrix},
 \label{eq:complex-spectral-lyapunov}
\end{equation}
with
\begin{equation}
 S_0(\theta)=
 \begin{pmatrix}1&\tilde q\\ \tilde q&1\end{pmatrix}.
 \label{eq:complex-initial}
\end{equation}

Realifying the complex variables by writing the state mode as $u+iv$ and the signal
mode as $a+ib$ gives the $4\times4$ covariance matrix
$\mathcal K_t(\theta)=\mathbb E[(u,v,a,b)^T(u,v,a,b)]$.  Equation
\eqref{eq:complex-spectral-lyapunov} is exactly \eqref{eq: K theta evolution}, with the
matrix $B_t(\theta)$ stated in Lemma \ref{lemma22}, and
\eqref{eq:complex-initial} becomes \eqref{eq:init}.  In this realification,
\begin{equation}
 S_{t,11}(\theta)=\mathcal K_{t,11}(\theta)+\mathcal K_{t,22}(\theta),
 \qquad
 \operatorname{Re}S_{t,12}(\theta)=\mathcal K_{t,13}(\theta)+\mathcal K_{t,24}(\theta).
 \label{eq:realification-identities}
\end{equation}

\begin{proof}[Proof of Lemma \ref{lemma22}]
The right-hand side of \eqref{eq: K theta evolution} is locally Lipschitz on
$\mathcal C(\mathbb S^1;\mathbb R^{4\times4})$, because $A_t$ and $C_t$ are bounded
linear functionals of $\mathcal K_t$ and the remaining dependence is polynomial.
Hence a unique maximal solution exists.  Positive semidefiniteness is preserved because,
for prescribed continuous $A_t,C_t$, \eqref{eq: K theta evolution} is a Lyapunov
covariance equation.

Put
\begin{equation}
 R_t=\frac1{2\pi}\int_{\mathbb S^1}
 \big(\mathcal K_{t,11}(\theta)+\mathcal K_{t,22}(\theta)\big)\,d\theta.
 \label{eq:spectral-mass}
\end{equation}
Taking the trace of the upper-left $2\times2$ block in
\eqref{eq: K theta evolution} and integrating gives
\begin{equation}
 \dot R_t=\big(2\beta A_t+2qC_t+\sigma_t^2\big)(1-R_t),
 \qquad R_0=1.
 \label{eq:spectral-mass-ode}
\end{equation}
Hence $R_t\equiv1$.  Positivity and Cauchy--Schwarz imply uniform bounds on $A_t$ and
$C_t$: from the definition of $A_t$, $|A_t|\leq R_t=1$, while the signal block has
constant integrated mass one, so $|C_t|\leq R_t^{1/2}=1$.  Therefore
$B_t(\theta)$ remains bounded on every finite time interval and the maximal solution
cannot blow up.  This proves global existence, while local Lipschitz uniqueness gives
uniqueness on $[0,T]$.
\end{proof}

\begin{proof}[Proof of Theorem \ref{Theorem Herglotz}]
By Lemma \ref{lemma:replica-chain},
\[
 A_t=\mathbb E[z_t^0z_t^{-1}],\qquad C_t=\mathbb E[z_t^0x^0].
\]
Applying \eqref{eq:matrix-herglotz} with $k=-1$ to the first component and with $k=0$
to the state--signal cross covariance gives
\begin{align}
 A_t
 &=\frac1{2\pi}\int_{-\pi}^{\pi}e^{-i\theta}S_{t,11}(\theta)\,d\theta
 =\frac1{2\pi}\int_{-\pi}^{\pi}\cos\theta\,S_{t,11}(\theta)\,d\theta,
 \label{eq:A-herglotz}\\
 C_t
 &=\frac1{2\pi}\int_{-\pi}^{\pi}\operatorname{Re}S_{t,12}(\theta)\,d\theta.
 \label{eq:C-herglotz}
\end{align}
Here $S_{t,11}$ is real and even because it is the scalar spectral density of the real
stationary sequence $(z_t^p)_{p\in\mathbb Z}$.  Using
\eqref{eq:realification-identities}, the right-hand sides of
\eqref{eq:A-herglotz}--\eqref{eq:C-herglotz} are exactly the quantities $A_t$ and $C_t$
defined in Lemma \ref{lemma22}.  This proves the theorem.
\end{proof}

\ACKNO{We are grateful to Benjamin Lindner for providing several useful references.}

\bibliographystyle{amsplain}
\bibliography{library,newbib,intro_additional_refs}

\end{document}